\documentclass[]{article}
\usepackage{geometry}
\newgeometry{hmargin={25mm,25mm}} 

\usepackage{times}
\usepackage{hyperref}
\usepackage{url}
\usepackage{amsmath}
\usepackage{amssymb}
\usepackage{graphicx}
\usepackage{tabularx}
\usepackage{makecell}
\usepackage{booktabs}
\usepackage{amsthm}
\usepackage{listings}
\usepackage{paralist}
\usepackage{color}
\usepackage{siunitx}
\usepackage{multirow}
\usepackage{bm}
\usepackage{algorithm}
\usepackage{algpseudocode}
\usepackage{authblk}
\usepackage{movie15}
\usepackage{caption}
\usepackage{siunitx}
\usepackage{nicefrac}
\DeclareMathOperator*{\relu}{ReLU}

\newtheorem{thm}{Theorem}

\DeclareSIUnit{\cal}{cal}
\DeclareSIUnit{\kcal}{\kilo\cal}
\DeclareCaptionLabelSeparator{colon}{}

\newcommand{\Ldyn}{L^{dyn}}
\newcommand{\Lorth}{L^{orth}}
\newcommand{\Rd}{\mathbb{R}^d}
\newcommand{\vect}[1]{\mathbf{#1}}

\newcommand{\expt}[1]{\mathbb{E}\left[#1\right]}
\newcommand{\prob}[1]{\mathbb{P}[#1]}

\title{A Data Driven Method for Computing Quasipotentials}
\author[1]{Bo Lin\footnote{Electronic mail: \href{mailto:E0046836@u.nus.edu}{E0046836@u.nus.edu}}}
\author[1,2]{Qianxiao Li\footnote{Corresponding author. Electronic mail: \href{mailto:qianxiao@nus.edu.sg}{qianxiao@nus.edu.sg}}}
\author[1]{Weiqing Ren\footnote{Corresponding author. Electronic mail: \href{mailto:matrw@nus.edu.sg}{matrw@nus.edu.sg}}}
\affil[1]{Department of Mathematics, National University of Singapore, Singapore 119076}
\affil[2]{Institute of High Performance Computing, A*STAR, Singapore 138632}
\date{\today}

\begin{document}

\maketitle

\begin{abstract}
    \noindent The quasipotential is a natural generalization of the concept of energy functions to non-equilibrium systems. In the analysis of rare events in stochastic dynamics, it plays a central role in characterizing the statistics of transition events and the likely transition paths. However, computing the quasipotential is challenging, especially in high dimensional dynamical systems where a global landscape is sought. 
    Traditional methods based on the dynamic programming principle or path space minimization tend to suffer from the curse of dimensionality.
    In this paper, we propose a simple and efficient machine learning method to resolve this problem. The key idea is to learn an orthogonal decomposition of the vector field that drives the dynamics, from which one can identify the quasipotential. We demonstrate on various example systems that our method can effectively compute quasipotential landscapes without requiring spatial discretization or solving path-space optimization problems. Moreover, the method is purely data driven in the sense that only observed trajectories of the dynamics are required for the computation of the quasipotential. These properties make it a promising method to enable the general application of quasipotential analysis to dynamical systems away from equilibrium. 
    
    \smallskip
    \noindent \textbf{Keywords.}  Non-equilibrium Systems, Quasipotential, Machine Learning, Rare Events, Hamilton-Jacobi Equations
\end{abstract}

\section{Introduction}
Dynamical systems under the influence of random perturbations are widely used in scientific modelling, including nucleation events during phase transitions, chemical reactions and biological networks. For these systems, understanding the mechanism and statistics of transitions between stable states is of great interest, especially when the noise has very small amplitude.
According to large deviation theory~\cite{freidlin2012random}, 
the transition dynamics become predictable in the small noise limit, and is completely characterized by the \emph{quasipotential}. The latter generalizes the notion of equilibrium potential to non-equilibrium systems. Consequently, the quasipotential landscape gives an intuitive description of the essential dynamical features of complex systems that are out of equilibrium~\cite{zhou2016construction,lv2015energy,nolting2016balls,lv2014constructing}. 

However, computing quasipotentials is a challenging problem, especially when the system is high dimensional, or when a global landscape is sought. 
To date, there are two classes of methods for computing the quasipotential. The first type relies on the variational formulation of the quasipotential based on the Freidlin-Wentzell action functional~\cite{weinan2004minimum,zhou2008adaptive,heymann2008geometric}. Here, the value of the quasipotential with respect to two chosen points is computed based on the solution of a path-space minimization problem. These methods have the advantage that they can handle high-dimensional systems, and moreover, a most likely transition path is identified together with the computation. However, the key disadvantage is the behavior of the quasipotential away from the chosen points (and a minimum action path connecting them) remains unknown. In particular, computing a quasipotential landscape is prohibitively expensive using such methods. 
The second class of methods is developed to compute the quasipotential on 2D or 3D meshes. These methods are based on the dynamic programming principle. At each step, the estimated quasipotential values at selected spatial points are updated by solving the associated Hamilton-Jacobi equation~\cite{cameron2012finding} or directly solving the action minimization problem locally~\cite{dahiya2018ordered,dahiya2018ordered2,yang2019computing}.
Contrasting with previous variational approaches in path space, these methods compute an entire quasipotential landscape and do not require {\textit{a priori}} information to select special points of interest. However, due to the requirement of a discretization mesh, they are limited to low dimensional system, as the computational complexity and cost grow exponentially over dimensions.

In practical applications, it is often the case that we need to analyze transition events or asymptotic occupational probabilities in high dimensional spaces, e.g. applications in biological networks~\cite{lv2015energy,li2018landscape,wang2010potential}. For such computations, the quasipotential is a very useful object.
Thus, it is of importance to develop a method that can effectively address the previously mentioned limitations. In this paper, we introduce a machine learning based method for computing quasipotential landscapes. The method is not only scalable to high dimensions, but also yields the entire quasipotential landscape. Moreover, it has the advantage that no explicit dynamical models are required, and the quasipotential can be constructed directly from sampled trajectory data. In fact, this method simultaneously learns the force field of the dynamical system from the trajectories. This makes the computation of quasipotential landscapes, thus the analysis of rare events, for practical applications a much more tractable task. 

The paper is organized as follows. We first introduce some theoretical background in Section~\ref{Background} and then propose the machine learning based method in Section~\ref{Methods}, including parameterization of the orthogonal decomposition of the vector field and the loss function. 
In Section~\ref{Numerical_Examples}, we illustrate the proposed method on several numerical examples. Finally we draw the conclusions in Section~\ref{Conclusion}.

\section{Background}\label{Background}
We consider a dynamical system driven by small white noise. Its evolution is described by the stochastic differential equation
\begin{equation}\label{SDE}
d\vect{x} = \vect{f}(\vect{x}) dt + \sqrt{\epsilon} d\vect{W},\quad \vect{x}=(x_1,\dots,x_D)\in\Rd,
\end{equation}
where $\vect{f}:\Rd \rightarrow \Rd$ is a continuously differentiable vector field, $\vect{W}$ is the standard Brownian motion and $\epsilon$ is a small parameter, typically identified as a scaled temperature.
For a given continuous path $\bm\varphi(t) \in \Rd$ on the time interval $t\in [0,T]$, the  Freidlin-Wentzell action functional of the path associated with the system is defined as
\begin{equation}
\mathcal{A}[\bm\varphi(\cdot);T] = \int_{0}^{T} \frac{1}{2}\left\lvert \dot{\bm{\varphi}}-\vect{f}(\bm\varphi)\right\rvert^2 dt.
\end{equation}
Denote by $\vect{x}^{\epsilon}(t)$ the trajectory of the system~\eqref{SDE} starting from $\bm\varphi(0)$. The Freidlin-Wentzell theory tells that for sufficiently small $\epsilon$, $\delta$, the probability that $\vect{x}^{\epsilon}(t)$ stays in the neighborhood of the path $\bm\varphi(t)$ on the time interval $[0,T]$ can be estimated by
\begin{equation}
\prob{\sup_{0\leq t\leq T}\left\lvert \vect{x}^{\epsilon}(t)-\bm\varphi(t)\right\rvert<\delta} \approx \exp(-\frac{1}{\epsilon}\mathcal{A}[\bm\varphi(\cdot);T]).
\end{equation}
We assume that the deterministic dynamical system $\dot{\vect{x}}=\vect{f}(\vect{x})$ exhibits a finite number of stable equilibria or limit cycles, such that almost every trajectory of the system is asymptotically convergent to those isolated attractors. Let $A$ be one of the attractors. The quasipotential at the state $\vect{x}$ with respect to the attractor $A$ is defined as
\begin{equation}\label{QP}
U_A(\vect{x}) = \inf_{T>0}\inf_{\bm{\varphi}} \mathcal{A}[\bm\varphi(\cdot);T],
\end{equation}
where the infimum of the action functional is taken over all time horizon $T>0$ and all absolutely continuous paths $\bm{\varphi}$ connecting the attractor $A$ and the state $\vect{x}$, {\it i.e.} $\bm\varphi(0)\in A$ and $\bm\varphi(T)=\vect{x}$. The quasipotential with respect to the attractor $A$ describes the difficulty of exiting the basin of $A$ for the system \eqref{SDE} when the strength of noise $\epsilon$ is small. According to the large deviation theory~\cite{freidlin2012random}, the statistics of the escaping event from the attractor $A$ can be estimated using the quasipotential. 
For instance, the maximum likelihood path from $A$ to another attractor is characterized by the quasipotential - the tangent of the path is parallel to $\vect{f}+\nabla U_A$ along the path. Also, the expected exit time $\tau$ from the attractor $A$ is determined by the minimum of the  quasipotential on the boundary of the basin of $A$: $\lim_{\epsilon\rightarrow0}\epsilon\log\expt{\tau}=\min_{\vect{x}\in\partial\mathcal{B}(A)}U_A(\vect{x})$, where $\mathcal{B}(A)$ is the basin of the attractor $A$.

The central idea of our approach relies on an alternative characterization of the quasipotential through an orthogonal decomposition of the vector field. Suppose $\vect{f}$ can be decomposed as
\begin{equation}\label{form}
\vect{f}(\vect{x})=-\nabla V(\vect{x})+\vect{g}(\vect{x}), \quad \text{with } 
\nabla V(\vect{x})^T\vect{g}(\vect{x})=0,
\end{equation}
where the term $-\nabla V(\vect{x})$ is referred to as the potential component of $\vect{f}(\vect{x})$ and $\vect{g}(\vect{x})$ as the rotational component. It is proved in the following theorem that under certain conditions, $2V$ coincides with the quasipotential of system~\eqref{SDE} up to an additive constant. 
\begin{thm}\label{thm1}
Suppose the vector field $\vect{f}$ in the system~\eqref{SDE} has the orthogonal decomposition~\eqref{form} and $V$ attains its strict local minimum at a point or limit cycle, denoted by $A$. 
If there is a bounded domain $\mathcal{D}$ containing $A$ such that
	\begin{itemize}
		\item $V$ is continuously differentiable in $\mathcal{D}\cup \partial \mathcal{D}$;
		\item $V(\vect{x})>V(A)$ and $\nabla V(\vect{x})\neq 0$ for all $\vect{x}\in \mathcal{D}\cup \partial \mathcal{D}$ and $\vect{x}\notin A$,
	\end{itemize}
then the quasipotential of the system~\eqref{SDE} with respect to the attractor $A$ in the set $\{\vect{x}\in \mathcal{D}\cup \partial \mathcal{D}: V(\vect{x})\leq\min_{\vect{y}\in\partial\mathcal{D}} V(\vect{y})\}$ coincides with $2V(\vect{x})$ up to an additive constant.
\end{thm}
\begin{proof}
    See Ref.~\cite{freidlin2012random}.
\end{proof}
For the system with multiple attractors, each attractor corresponds to a local quasipotential. These local quasipotential can be used to construct the global quasipotential~\cite{freidlin2012random,zhou2016construction,bouchet2016perturbative}. The global quasipotential is related to the invariant measure of the dynamical system when the noise is small: $\lim_{\epsilon\rightarrow 0}\epsilon\log p_{\infty}(\vect{x}) =-U(\vect{x})$, where $p_{\infty}(\vect{x})$ is the steady-state probability distribution of the system.

\section{Methods}\label{Methods}
We construct the quasipotential based on the orthogonal decomposition~\eqref{form}, where the potential and rotational components are parameterized by neural networks. 
For systems with multiple attractors, we use a single neural network for the potential component and a single neural network for the rotational component across the whole domain of interest. The local quasipotential with respect to each attractor can be obtained by confining the parameterized function to the corresponding basin of attraction.

Once we have a suitable parameterization of $V$ and $\vect{g}$, they can then be trained by minimizing a loss function over the trajectory data from the deterministic system
\begin{equation}\label{det_sys}
    \dot{\vect{x}} = \vect{f}(\vect{x}) = -\nabla V(\vect{x}) + \vect{g}(\vect{x}).
\end{equation} 
The loss function is designed to reconstruct the dynamics of the original system~\eqref{det_sys} and to impose the orthogonality condition between the potential and rotational components.
\subsection{Parameterization of the Orthogonal Decomposition}
The function $V$ is parameterized by the sum of a neural network and a quadratic function,
\begin{equation}\label{par_V}
V_{\theta}(\vect{x}) = \hat{V}_{\theta}(\vect{x}) +  \lvert\vect{x}\rvert^2,
\end{equation}
where the activation function of the network $\hat{V}_{\theta}$ is taken as the hyperbolic tangent function.
The rotational component $\vect{g}$ is parameterized by a neural network $\vect{g}_{\theta}$ with continuously differentiable activation (e.g. $\tanh(z)$ or $\relu^2(z)$~\cite{li2019better}). Therefore, the parameterized vector field is
\begin{equation}\label{par_f}
    \vect{f}_{\theta}(\vect{x}) = -\nabla V_{\theta}(\vect{x}) + \vect{g}_{\theta}(\vect{x}).
\end{equation}
The neural networks $V_{\theta}(\vect{x})$ and $\vect{g}_{\theta}(\vect{x})$ are constructed to obey the following properties:
\begin{itemize}
    \item[(i)] $V_{\theta}$ is real analytic;
    \item[(ii)] Both $V_{\theta}$ and $\lvert\nabla V_{\theta}\rvert$ are radially unbounded, i.e. $V_{\theta}(\vect{x})\rightarrow\infty$ and 
         $\lvert\nabla V_{\theta}(\vect{x})\rvert\rightarrow\infty$, 
         as {\small $\lvert\vect{x}\rvert\rightarrow\infty$};
    \item[(iii)] $\vect{g}_{\theta}$ is continuously differentiable.
\end{itemize}
The following theorem shows that under the above three conditions, the set $\{\vect{x}\in \Rd:\nabla V_{\theta}(\vect{x})=0\}$ is bounded and has Lebesgue measure zero in $\Rd$, and the learned dynamics $\dot{\vect{x}} = \vect{f}_{\theta}(\vect{x})$ is stable with respect to this set.
Hence, any dynamics parameterized as such enjoys good stability properties, and are suitable candidates to model physical systems.

\begin{thm} \label{thm2}
Let $\vect{f}(\vect{x})=-\nabla V(\vect{x}) + \vect{g}(\vect{x})$ where $\nabla V(\vect{x})^T\vect{g}(\vect{x})=0$ and $V$, $\vect{g}$ satisfy the conditions (i),(ii),(iii). Then any trajectory $\{\vect{x}(t)\}_{t\geq0}$ of the system $\dot{\vect{x}} = \vect{f}(\vect{x})$ approaches the bounded measure-zero set $\mathcal{C}:=\{\vect{x}\in \Rd:\nabla V(\vect{x})=0\}$ as $t\rightarrow\infty$, {\it i.e.}
	\begin{equation}
	    \lim_{t\rightarrow\infty} \inf_{\vect{y}\in\mathcal{C}}
	    \lvert \vect{x}(t)-\vect{y}\rvert=0.
	\end{equation}
\end{thm}
\begin{proof}
As $V$ is real analytic, all partial derivatives of $V$ are also real analytic. Since $V$ is radially unbounded, the zero sets of these partial derivatives are all measure-zero in $\Rd$. Thus, the set $\mathcal{C}$ has measure of zero in $\Rd$. Furthermore, $\lvert \nabla V(\vect{x})\rvert\rightarrow \infty$, as $\lvert \vect{x}\rvert\rightarrow\infty$, which implies that $\mathcal{C}$ is also bounded.

For any trajectory $\{\vect{x}(t)\}_{t\geq0}$ of the system $\dot{\vect{x}} = \vect{f}(\vect{x})$, we have
\begin{equation}
\frac{dV(\vect{x}(t))}{dt} 
= \nabla V(\vect{x}(t))\cdot\vect{f}(\vect{x}(t)) = -\lvert\nabla V(\vect{x}(t))\rvert^2\leq 0.
\end{equation}
Therefore $V$ is the Lyapunov function of this system and we have $V(\vect{x}(t))\leq V(\vect{x}(0))$, for all $t$. 
Furthermore, since $V$ is radially unbounded, the sub-level set
\begin{equation}
    S_0=\{\vect{x}\in \Rd: V(\vect{x})\leq V(x(0))\}
\end{equation} 
is bounded. The trajectory $\{\vect{x}(t)\}_{t\geq0}$ is contained in the bounded set $S_0$. By Lasalle's theorem~\cite{lasalle1960some}, the trajectory $\{\vect{x}(t)\}_{t\geq0}$ approaches the set $\mathcal{C}$ as $t\rightarrow\infty$.
\end{proof}

\paragraph{Remark.}
\textit{Incidentally, the data-driven nature of our method also gives a way to learn stable and interpretable dynamical systems from trajectory data, as shown in Theorem~\ref{thm2}.
Up to this paper, a large amount of efforts have been devoted to the various data-driven methods for system identification in two main directions. One is to learn closed form equations with some prior knowledge on the underlying mechanism. Related methods include Kronecker product representations~\cite{yao2007modeling}, sparse identification of nonlinear dynamics~\cite{brunton2016discovering}, Gaussian processes~\cite{raissi2018hidden} and PDE-net~\cite{long2019pde}. The other direction employs black box methods to learn a model with better accuracy in the prediction.
These methods exploit the expressive power of deep neural networks~\cite{raissi2018multistep,raissi2018deep}, which could potentially learn more complicated models of the nonlinear dynamical systems. However, stability and interpretability is not generally ensured.
The attempts to balance expressive power and physical relevance is investigated in~\cite{yu2020onsagernet,kolter2019learning}. The current method falls into this category, in that stability is ensured by construction, and subsequent flexibility is introduced via neural network approximation.}

\subsection{Loss Function}
Once we have parameterized $V_{\theta}$ and $\vect{g}_{\theta}$, it remains to define a suitable loss function over the data in order to train them to ensure reconstruction ($\vect{f} \approx - \nabla V_{\theta} + \vect{g}_{\theta}$) and orthogonality ($\nabla V_{\theta}^T \vect{g}_{\theta} \approx 0$).

The observation dataset $X=\{X_i(t_j),X_i(t_j+\Delta t):\ i=1,\dots,N,\ j=0,\dots,M-1\}$ contains $N$ trajectories of the deterministic system~\eqref{det_sys} where $X_i(t)$ denotes the $i^\text{th}$ trajectory. Along each trajectory, $2M+2$ data points are sampled at the times
\begin{equation}\label{times}
t_0,t_0+\Delta t,t_1,t_1+\Delta t,...,t_M,t_M+\Delta t,
\end{equation}
where $t_0<t_1<...<t_M$ and $\Delta t$ is a small time step. The loss function consists of two parts
\begin{equation}
L = \Ldyn +  \lambda \Lorth,
\end{equation}
where $\Ldyn$ is to reconstruct the dynamics in~\eqref{det_sys}, $\Lorth$ is to impose the orthogonality condition $\nabla V_{\theta}(\vect{x})^T\vect{g}_{\theta}(\vect{x})=0$, and $\lambda$ is a parameter. 

The term $\Ldyn$ depends on the difference between the learned dynamics and the observed trajectories,
\begin{equation}
\begin{aligned}
\Ldyn &= \frac{1}{N(M+1)}\sum_{i=1}^N \sum_{j=0}^{M}
\bar{h}\left( \vect{e}_{ij} ;\delta_1\right),\\
\vect{e}_{ij} &= \frac{1}{\Delta t}\left(\mathcal{I}_{\Delta t}[\vect{f}_{\theta};X_i(t_j)]-X_i(t_j+\Delta t)\right),
\end{aligned}
\end{equation}
where $\mathcal{I}_{\Delta t}[\vect{f}_{\theta};X_i(t_j)]$ is the state obtained by performing the numerical integration of the learned dynamics $\dot{\vect{x}}=\vect{f}_{\theta}(\vect{x})$ by one time step $\Delta t$ from the state $X_i(t_j)$, and $\bar{h}(\vect{e};\delta_1)$ denotes the mean Huber loss of the vector $\vect{e}=(e_1,\dots,e_d)$ with threshold $\delta_1$,
\begin{equation}
\begin{aligned}
\bar{h}(\vect{e};\delta_1) &= \frac{1}{d}\sum_{i=1}^{d} h(e_i;\delta_1),\\
h(e_i;\delta_1) &= \left\{
\begin{array}{ll}
\frac{1}{2} e_i^2, & \lvert e_i\rvert<\delta_1,\\
\delta_1\lvert e_i\rvert -\frac{1}{2}\delta_1^2, & \text{otherwise}. \\
\end{array} \right.
\end{aligned}
\end{equation} 
The Huber loss reduces the dominating effect of large components in the vector $\vect{e}$.

The orthogonality between $\nabla V_{\theta}$ and $\vect{g}_{\theta}$ is imposed by the penalty term $\lambda\Lorth$ with 
\begin{equation}
\Lorth =  \frac{1}{S}\sum_{i=1}^S  w\left(\frac{\nabla V_{\theta}(\tilde{X}_i)^T \vect{g}_{\theta}(\tilde{X}_i) }{\lvert \nabla V_{\theta}(\tilde{X}_i)\rvert \cdot\lvert\vect{g}_{\theta}(\tilde{X}_i)\rvert};\delta_2\right),
\end{equation}
where $w(y;\delta_2)=y^2I_{y>0}+ \delta_2 y^2I_{y<0}$, $\delta_2$ is a parameter and
$\tilde{X}_1,\dots,\tilde{X}_S$ are representative data points sampled from $X$ by using Algorithm~\ref{alg1}. The representative data points are chosen such that each of them covers a ball of radius $r$ and no other representative data points lie inside this ball. This sampling procedure avoids the situation where points in the trajectories are clumped together near attractors, where the orthogonality condition is difficult to enforce numerically.
\begin{algorithm}
	\caption{Sampling the representative dataset}
	\label{alg1}
	\begin{algorithmic}[1]
		\Function{GetXhat}{$X$, $r$}
		\State Initialize the sets $Y=X$ and $\tilde{X}=\emptyset$
		\State\textbf{while} $Y\neq\emptyset$ \textbf{do}
		\State\hspace{13pt} Randomly select $\vect{x}\in Y$ and append $\vect{x}$ to the set $\tilde{X}$
		\State\hspace{13pt} Delete all the points belonging to the ball $B_{r}(\vect{x})$ from $Y$
		\State\textbf{end while}
		\State Return $\tilde{X}$
		\EndFunction
	\end{algorithmic}
\end{algorithm}

\section{Numerical Examples}\label{Numerical_Examples}

We now illustrate using various numerical examples that the proposed method can efficiently compute the quasipotential and at the same time learn stable dynamics.
Section~\ref{example12} contains two ODE systems: one with two stable equilibrium points and the other with a limit cycle. The quasipotentials are known in these two examples, and we use these exact solutions to benchmark the numerical method.
Section~\ref{example3} is a biological system which models the reproduction process of a budding yeast cell cycle.
Section~\ref{example45} contains two high-dimensional systems which are obtained from the discretization of partial differential equations (PDEs).

In the examples, we generate trajectories by simulating the deterministic dynamics in Eq.~\eqref{det_sys} using the forth-order Runge-Kutta method with the time step $\Delta t$ on the time interval $[0,T]$. The initial states are randomly sampled from certain distributions which will be specified in the examples. From these trajectories, we obtain the dataset $X$ by collecting the data points at the times $t_j=j m\Delta t$ and $t_j+\Delta t$ where $j=0,1,\dots,M$ and $m$ is some positive integer. The set of trajectories is split into three parts: $70\%$ (training), $20\%$ (validation) and $10\%$ (test). The representative datasets are sampled from these three datasets respectively using Algorithm~\ref{alg1} with various choices of the parameter $r$. The parameters $\Delta t$, $T$, $m$, $r$ and the number of trajectories $N$ are given in Table~\ref{tab1}.
\begin{table}[h]
	\caption{: Parameters in the numerical examples.}
	\label{tab1}
	\begin{center}
		\begin{tabular}{ ccccc ccc c }
			\hline\hline\vspace{-0.25cm}\\
			Example & $N$ & $\Delta t$ & $T$  & $m$ & \makecell{$r$} & $\delta_1$ & $\lambda$  &
			\makecell{ \# nodes in each \\ hidden layer} \vspace{0.1cm}\\
			\hline\hline \vspace{-0.25cm}\\
			1 & $2\times10^3$ & $10^{-2}$ &$5$ & $10$ & $0.1$ & $1$ & $1$  & $50$ \vspace{0.1cm}\\
			\hline \vspace{-0.25cm}\\
			2 & $2\times10^3$ & $10^{-2}$ &$5$ & $10$ & $0.05$ & $1$  & $0.02$  & $50$ \vspace{0.1cm}\\
			\hline \vspace{-0.25cm}\\
			3 & $1\times10^4$ & $10^{-2}$ &$50$ & $100$ & $0.1$ & $1$  & $0.005$  & $100$  \vspace{0.1cm}\\
			\hline \vspace{-0.25cm}\\
			4 & $1\times10^4$ & $10^{-3}$ &$2$ & $20$ & $0.2$ & $1$ & $1$  & $100$ \vspace{0.1cm}\\
			\hline \vspace{-0.25cm}\\
			5 & $2\times10^4$ & $10^{-4}$ &$2$ & $200$ & $0.2$ &$1$ & $0.1$ & $200$ \vspace{0.1cm}\\
			\hline
		\end{tabular}
	\end{center}
\end{table}

The networks $\hat{V}_{\theta}$, $\vect{g}_{\theta}$ for the potential and rotational components in the parameterized vector field~\eqref{par_f} are both taken as fully connected neural networks of $2$ hidden layers with the same number of nodes in each hidden layer.
The nonlinear activation function in $\hat{V}_{\theta}$ is $\tanh$ in all the examples, and the activation function in $\vect{g}_{\theta}$ is $\tanh$ in Examples 1-3 and $\relu^2$ in Examples 4-5.
The input to the parameterized vector field $\vect{f}_{\theta}$ is centered so that the centered data points have mean-zero.

In the loss function, we use the second-order Runge-Kutta method as the numerical integrator $\mathcal{I}$ and set $\delta_2=\frac{1}{10}$.
The two parameters $\delta_1$, $\lambda$ are chosen so that the orthogonality error $\Lorth$ and the error of the predicted long-term dynamics over the test dataset are both small. To quantify the accuracy of the predicted long-term dynamics, we solve the learned dynamics
\begin{equation}\label{learned_dym}
\dot{\vect{x}}_{\theta} = -\nabla V_{\theta}(\vect{x}_{\theta}) + \vect{g}_{\theta}(\vect{x}_{\theta})
\end{equation}
using the second-order Runge-Kutta method on the time interval $[0,T]$, and compare the solution with the original dynamics:
\begin{equation}\label{error}
\epsilon = 
\frac{
    \sqrt{ \sum_{j=1}^M \left\lvert\vect{x}_{\theta}(t_j)-\vect{x}(t_j)\right\rvert^2 }
}
{
    \sqrt{ \sum_{j=1}^M \left\lvert\vect{x}(t_j)\right\rvert^2}
},
\end{equation}
where $\vect{x}(t)$ is the trajectory from the test dataset with the same initial sate as $\vect{x}_{\theta}(t)$.

The networks are trained with Adam optimizer~\cite{kingma2014adam} using mini-batches of size $5000$, while the learning rate exponentially decays over the training steps.

\subsection{ODE systems with known quasipotentials}\label{example12}
First, we consider two low-dimensional systems: one with two stable equilibrium points and the other with a limit cycle. The quasipotentials are known in these two examples, and we use these exact quasipotentials to benchmark the proposed method.
~\\~\\
\noindent\textit{Example 1.} We consider the following system in three-dimensional space~\cite{yang2019computing},
\begin{equation}\label{example1_ODE}
\begin{aligned}
\frac{dx}{dt} &= -2(x^3-x)-(y+z),\\
\frac{dy}{dt} &= -y+2(x^3-x),\\
\frac{dz}{dt} &= -z+2(x^3-x),\\
\end{aligned}
\end{equation}
where the state of the system is $\vect{x}=(x,y,z)^T$. This system has two stable equilibrium points, one at $\vect{x}_a=(-1,0,0)$ and the other at $\vect{x}_b=(1,0,0)$ and one unstable equilibrium point at $\vect{x}_c=(0,0,0)$. In the basins of the two stable equilibrium points, the quasipotential is known and given by
\begin{equation}\label{example1_QP}
U(x,y,z) = (1-x^2)^2+y^2+z^2.
\end{equation}

We generate $2000$ trajectories by solving the equations in~\eqref{example1_ODE} starting from initial states sampled from the uniform distribution on the domain $\mathcal{D}=[-2,2]\times[-1.5,1.5]^2$. Along each trajectory, we collect $100$ data points. In total, $X$ contains $2\times10^5$ data points. Out of these data points, $8571$ representative data points are used to impose the orthogonality condition.

The test dataset contains $200$ trajectories. To quantify the accuracy of the predicted long-term dynamics, we solve the learned dynamics starting from the initial states of these trajectories and compute the error in~\eqref{error} for each trajectory. Fig.~\ref{fig1} (lower panel) shows the comparison of three trajectories of the learned dynamics with those of the original dynamics in the test dataset. These errors have the mean $5.069\times 10^{-4}$ and the standard deviation $1.565\times 10^{-3}$.

The learned quasipotential is given by $U_{\theta}(\vect{x})=2V_{\theta}(\vect{x})-C$, where the constant $C$ is such that the minimum of $U_{\theta}(\vect{x})$ on the domain $\mathcal{D}$ equals zero. Fig.~\ref{fig1} (upper panel) shows the comparison of $U_{\theta}(\vect{x})$ with the exact quasipotential in~\eqref{example1_QP}. To quantify the accuracy of the learned quasipotential, we compute the relative root mean square error (rRMSE) and the relative mean absolute error (rMAE), 
\begin{equation}\label{rRMSE_rMAE}
\begin{aligned}
\text{rRMSE} =
\frac{
   \sqrt{\sum_{i=1}^L\left(U(\vect{x}_i)-U_{\theta}(\vect{x}_i)\right)^2}
}
{
    \sqrt{\sum_{i=1}^L U^2(\vect{x}_i)}
}, \quad 
\text{rMAE} =
\frac{
   \sum_{i=1}^L\left\lvert U(\vect{x}_i)-U_{\theta}(\vect{x}_i)\right\rvert
}
{
   \sum_{i=1}^L\left\lvert U(\vect{x}_i)\right\rvert
}, 
\end{aligned}
\end{equation}
where $\{\vect{x}_i\}_{i=1}^L$ are the grid points of the uniform mesh on $\mathcal{D}$. The rRMSE and rMAE for the learned quasipotential are $0.0037$ and $0.0017$, respectively. The errors are computed with $L=10^6$.

\begin{figure}[t!]
	\centering
	\includegraphics[width=\linewidth]{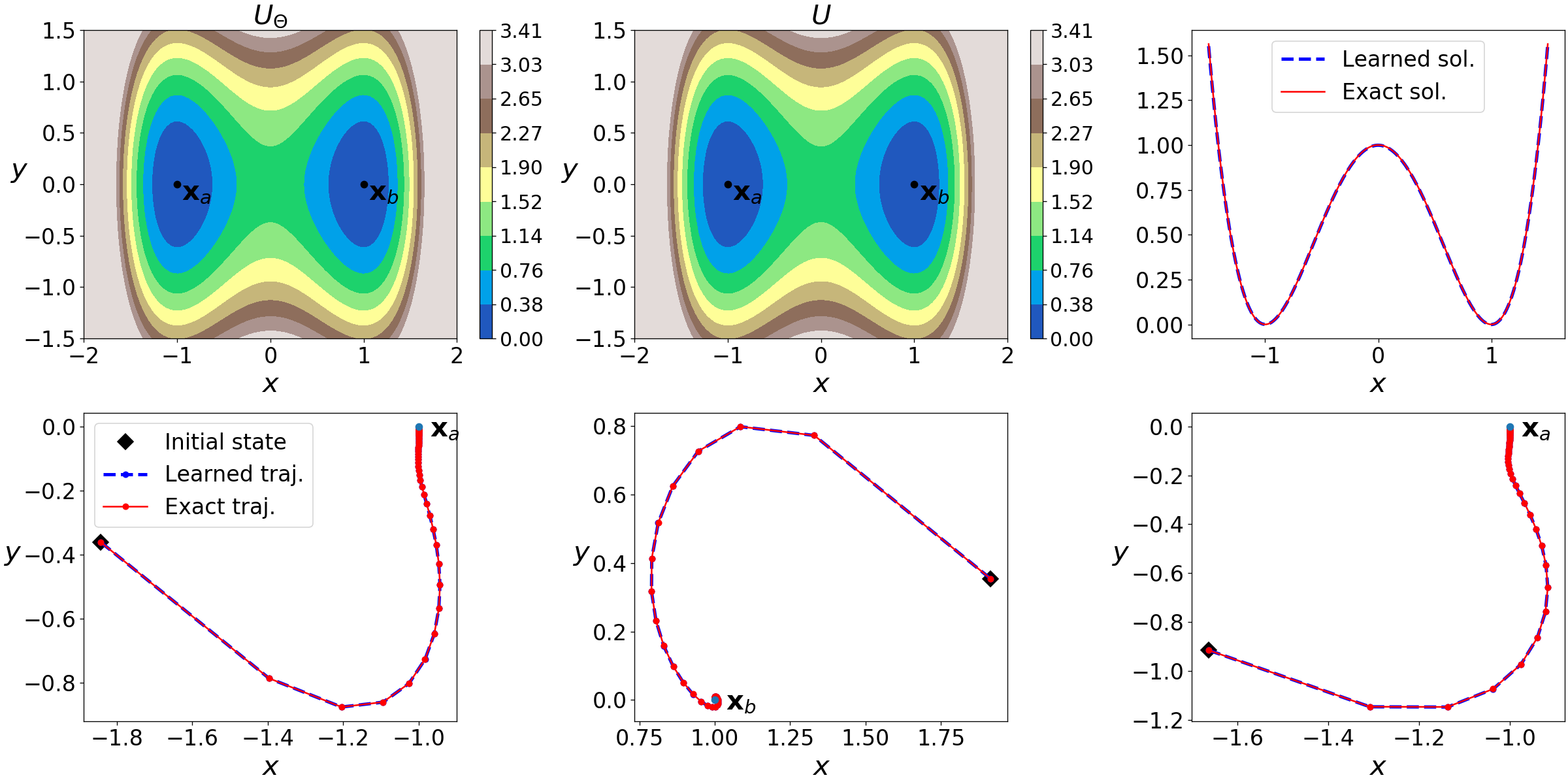}
	\caption{\ (Example 1): \textit{Upper Panel}: Contour plots of the learned quasipotential $U_{\theta}$ (left) and exact quasipotential $U$ (middle) projected onto the $xy$ plane with $z=0$, and plot of the learned quasipotential along the line $y=z=0$ (right). 
	\textit{Lower Panel}: Comparison of trajectories of the learned dynamics and the original dynamics~\eqref{example1_ODE} from different initial states.}
	\label{fig1}
\end{figure}
~\\
\noindent\textit{Example 2.} We consider the system with the quasipotential
\begin{equation}\label{example2_QP}
U(x,y) = \left((x-a)^2+(x-a)(y-b)+(y-b)^2-\frac{1}{2}\right)^2,\quad (x,y)\in R^2,
\end{equation}
where $a$, $b$ are two parameters. The function $U$ attains its local maximum at the point $(a,b)$ and attains its minimum on the ellipse
\begin{equation}\label{limitcycle}
\left\{(x,y)\in R^2:(x-a)^2+(x-a)(y-b)+(y-b)^2=\frac{1}{2}\right\}.
\end{equation}
The dynamics for the system is governed by
\begin{equation}\label{example2_ODE}
\begin{aligned}
\frac{dx}{dt} &= -\frac{1}{2} \frac{\partial U}{\partial x}(x,y)-2 \left( x+2y-a-2b\right), \\
\frac{dy}{dt} &= -\frac{1}{2} \frac{\partial U}{\partial y}(x,y)+2 \left( 2x+y-2a-b\right),
\end{aligned}
\end{equation}
where the state of the system is $\vect{x}=(x,y)^T$. This dynamical system has a stable limit cycle on the ellipse in~\eqref{limitcycle} and an unstable equilibrium point at $(a,b)$ inside the limit cycle. 

We take $a=1$, $b=2.5$ and generate $2000$ trajectories by solving the equations in~\eqref{example2_ODE} starting from initial states sampled from the uniform distribution on the domain $\mathcal{D} = [-0.5,2.5]\times[1,4]$. Along each trajectory, we collect $100$ data points. In total, $X$ contains $2\times 10^5$ data points. Out of these data points, $3712$ representative data points are used to impose the orthogonality condition.

Fig.~\ref{fig2} (lower panel) shows a comparison of one trajectory of the learned dynamics with that of the original dynamics in the test dataset. The statistics (mean $\pm$ deviation) of the errors of $200$ trajectories is $4.797\times 10^{-4} \pm 2.923\times 10^{-4}$.
A comparison of the learned quasipotential with the exact quasipotential in~\eqref{example2_QP} is shown in Fig.~\ref{fig2} (upper panel).
The rRMSE and rMAE for the learned quasipotential on the domain $\mathcal{D}$ are $0.0141$ and $0.0090$, respectively. The errors are computed using Eqs.~\eqref{rRMSE_rMAE} with $L=10^4$.

\begin{figure}[t!]
	\centering
	\includegraphics[width=\linewidth]{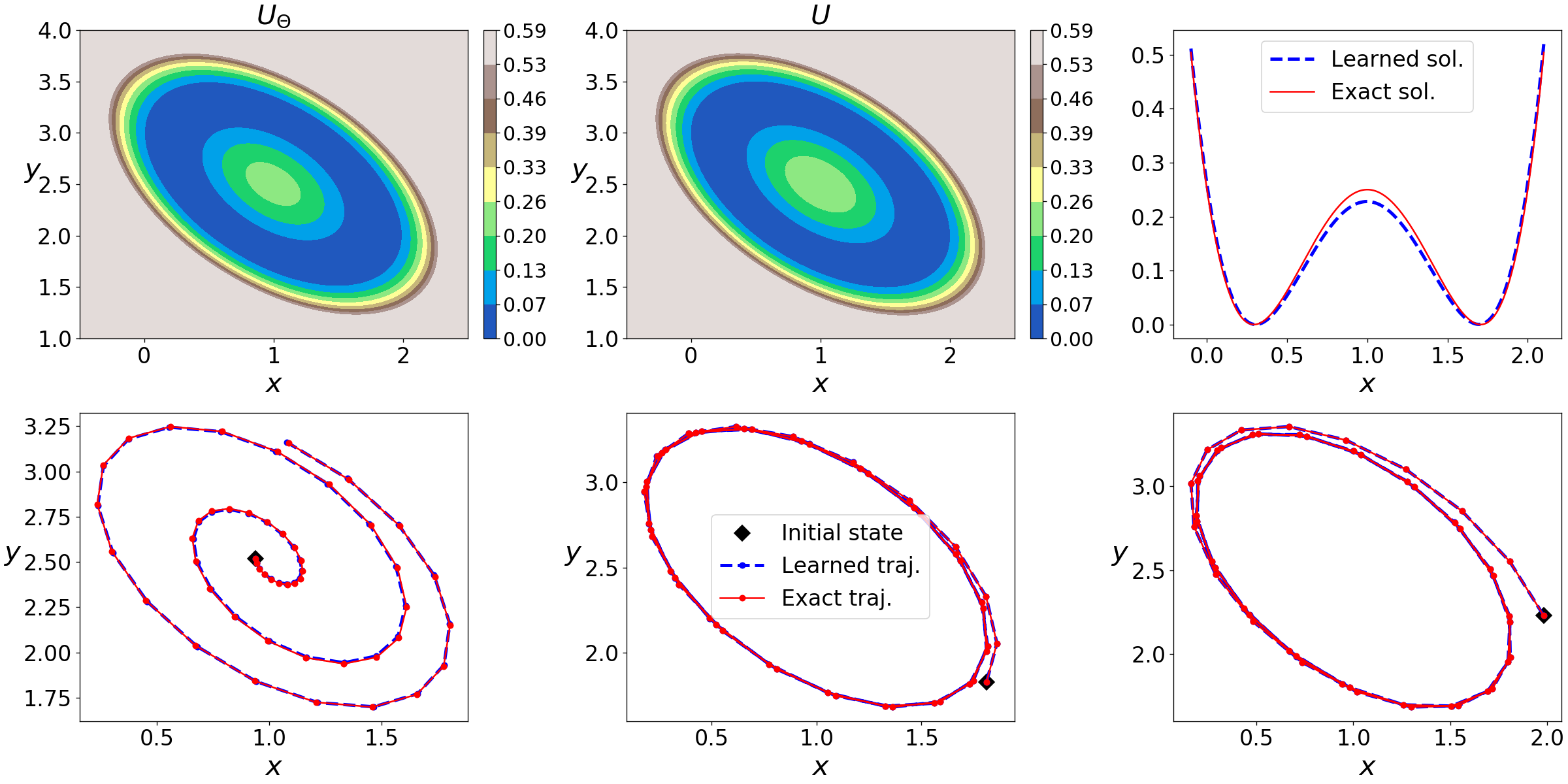}
	\caption{\ (Example 2): \textit{Upper Panel}: Contour plots of the learned quasipotential $U_{\theta}$ (left) and exact quasipotential $U$ (middle) and plot of the learned quasipotential along the line $y=b$ (right). 
	\textit{Lower Panel}: Comparison of trajectories of the learned dynamics and the original dynamics~\eqref{example2_ODE} from different initial states.}
	\label{fig2}
\end{figure}

\subsection{Biological system: budding yeast cell cycle}\label{example3}
The previous two examples are toy problems where the exact quasipotential is known. Now, we test our method on a more challenging problem where computing quasipotential landscapes using traditional methods may be very expensive.
~\\~\\
\noindent\textit{Example 3.} 
We study the robustness of the reproduction process of a budding yeast cell cycle by constructing the quasipotential~\cite{lv2015energy}. 
The simplified network of yeast cell is composed of three modules: the $G1/S$ module, the early $M$ module and the late $M$ module.
Based on the feedback of each module and the interactions between different modules, the following dynamics has been proposed for the cell cycle
\begin{equation}\label{example3_ODE}
\begin{aligned}
\frac{dx}{dt} &= \frac{x^2}{j_1^2+x^2}-k_1 x-xy+a_0,\\
\frac{dy}{dt} &= \frac{y^2}{j_2^2+y^2}-k_2 y-yz+k_{a1}x,\\
\frac{dz}{dt} &= \frac{k_sz^2}{j_3^2+z^2}-k_3 z-k_izx+k_{a2}y,\\
\end{aligned}
\end{equation}
where $x$, $y$, $z$ represent the concentration of certain key regulators in the $G1/S$, early $M$ and late $M/G1$ phase, respectively. The values for the parameters $j_1$,$j_2$,$j_3$,$k_1$,$k_2$,$k_3$,$k_i$,$k_s$,$k_{a1}$,$k_{a2}$,$a_0$ are taken from Ref.~\cite{lv2015energy}. The dynamics has a stable equilibrium state $G1$ approximately at $(0,0,z_{max})$ where $z_{max}=4.342$. 
The yeast cell cycle is termed a \emph{robust process} in~\cite{lv2015energy}, in the sense that most transition paths stay close to a particular pathway due to the dynamical landscape.
This pathway starts from the excited $G1$ state and ends at the stable $G1$ state by going through the $S$ phase approximately at $(x_{max},0,0)$ where $x_{max}=4.335$ and the early $M$ state approximately at $(0,y_{max},0)$ where $y_{max}=4.353$.

We generate $10^4$ trajectories by solving the equations in~\eqref{example3_ODE} starting from initial states sampled from the uniform distribution on the set
\begin{equation}
\{\vect{x}=(x,y,z)\in [0,5]^3: \left\lVert \vect{f}(\vect{x})\right\rVert_{\infty}<5\},
\end{equation}
where the notation $\left\lVert\vect{y}\right\rVert_{\infty}$ denotes the maximum of absolute values of the 
components in the vector $\vect{y}$. The last condition excludes states far away from the regions of interest corresponding to the transition events. Along each trajectory, we collect $100$ data points. In total, $X$ contains $10^6$ data points. Out of these data points, $8384$ representative data points are used to impose the orthogonality condition.

Fig.~\ref{fig3} (lower panel) shows a comparison of one trajectory of the learned dynamics and that of the original dynamics in the test dataset. The statistics (mean $\pm$ deviation) of the errors of the $1000$ trajectories is $0.161\pm 0.226$.
The cross-sections of the learned quasipotential at $z=0$ and $x=0$ are shown in Fig.~\ref{fig3} (upper panel). The quasipotential characterizes the robust process of the cell cycle, which agrees well with the result in Ref.~\cite{lv2015energy} using the geometric minimum action method. Moreover, notice that the quasipotential we compute can be evaluated at arbitrary points in space (in the regions explored by the sampled data) and is not limited by any meshes, or choice of beginning and end points for path-based methods.

\begin{figure}[t!]
	\centering
	\includegraphics[width=.8\linewidth]{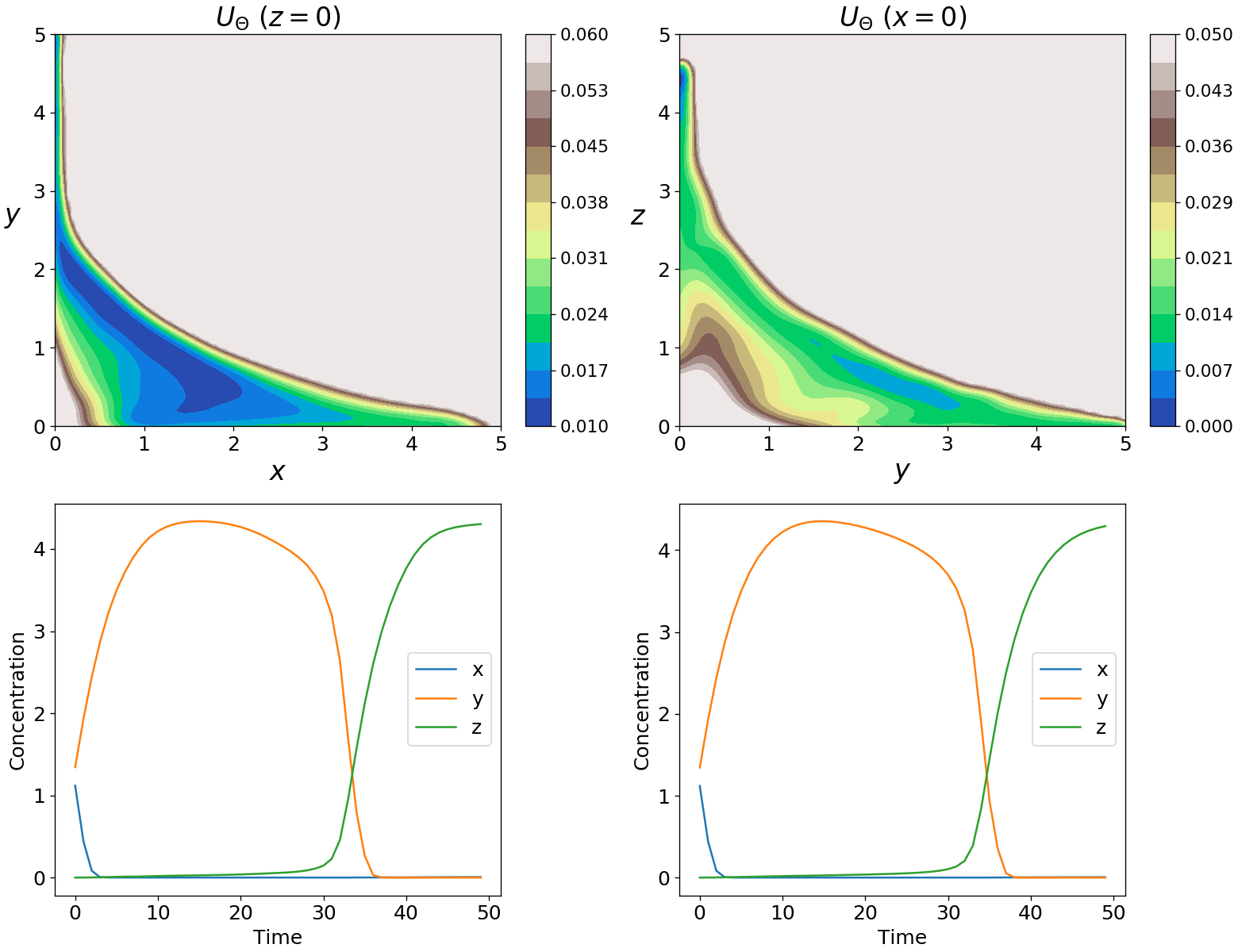}
	\caption{\ (Example 3): \textit{Upper Panel}: Contour plots of the quasipotential projected onto the $xy$-plane with $z=0$ (left) and the $yz$-plane with $x=0$ (right). \textit{Lower Panel}: Comparison of one trajectory of the learned dynamics (left) and the original dynamics in~\eqref{example3_ODE} (right).
		}
	\label{fig3}
\end{figure}

\subsection{High-dimensional systems: discretized PDEs}\label{example45}
We next apply the proposed method to two high-dimensional systems which are obtained from the discretization of PDEs. After discretization, the first system is a gradient system in the $50$-dimensional space with known quasipotential, and the second one is a non-gradient system in the $40$-dimensional space.
~\\~\\
\noindent\textit{Example 4.} We consider the Ginzburg-Landau equation
\begin{equation}\label{GL}
u_t = \delta u_{xx} -\delta^{-1} V'(u),\quad x\in[0,1],
\end{equation}
with the boundary conditions $u(0,t)=u(1,t)=0$ and the initial condition $u(x,0)=u^0(x)$, where $V(u)=\frac{1}{4}(1-u^2)^2$ is the double-well potential and $\delta$ is a small parameter. The equation is a gradient flow associated with the energy
\begin{equation}\label{E4_cont}
E[u] = \int_{0}^1 \left(\frac{1}{2}\delta u_x^2+\delta^{-1} V(u)\right) dx.
\end{equation}

We partition the interval $[0,1]$ using $I+1$ grid points $x_0$,...,$x_I$, where $x_i=ih$ and $h=1/I$.
Then we approximate the spatial derivatives in Eq.~\eqref{GL} using the central finite difference and obtain the following system of ODEs
\begin{equation}\label{disct_GL}
\frac{d u_i}{dt} = \delta \frac{u_{i-1}-2u_i+u_{i+1}}{h^2}-\delta^{-1} V'(u_i),\quad 1\leq i\leq I-1,
\end{equation}
with $u_0=u_I=0$ and the initial condition $u_i(0)=u^0(x_i)$ for $1\leq i\leq I-1$, where $u_i$ denotes the approximate solution at the grid point $x_i$. The state of the system is denoted by $\vect{u}=(u_1,\dots,u_{I-1})$.
The ODE system is a gradient flow associated with the energy
\begin{equation}\label{disct_E}
E_h[\vect{u}] = \sum_{i=1}^I\frac{1}{2}\delta\left(\frac{u_i-u_{i-1}}{h}\right)^2+\delta^{-1}V(u_i),
\end{equation}
which is a discretization of the energy~\eqref{E4_cont}, up to the factor $h$.
The dynamics~\eqref{disct_GL} has two stable states at the two local minima $\vect{u}_{\pm}$ of the energy~\eqref{disct_E}, which are shown in Fig.~\ref{fig5} (last column) for $\delta=0.1$. The quasipotential with respect to the two stable states is 
\begin{equation}\label{QP_example4}
    U(\vect{u})=2E_h[\vect{u}] + C
\end{equation} 
in the basins of attraction, where $C$ is constant.

\begin{figure}[t!]
	\centering
	\includegraphics[width=\linewidth]{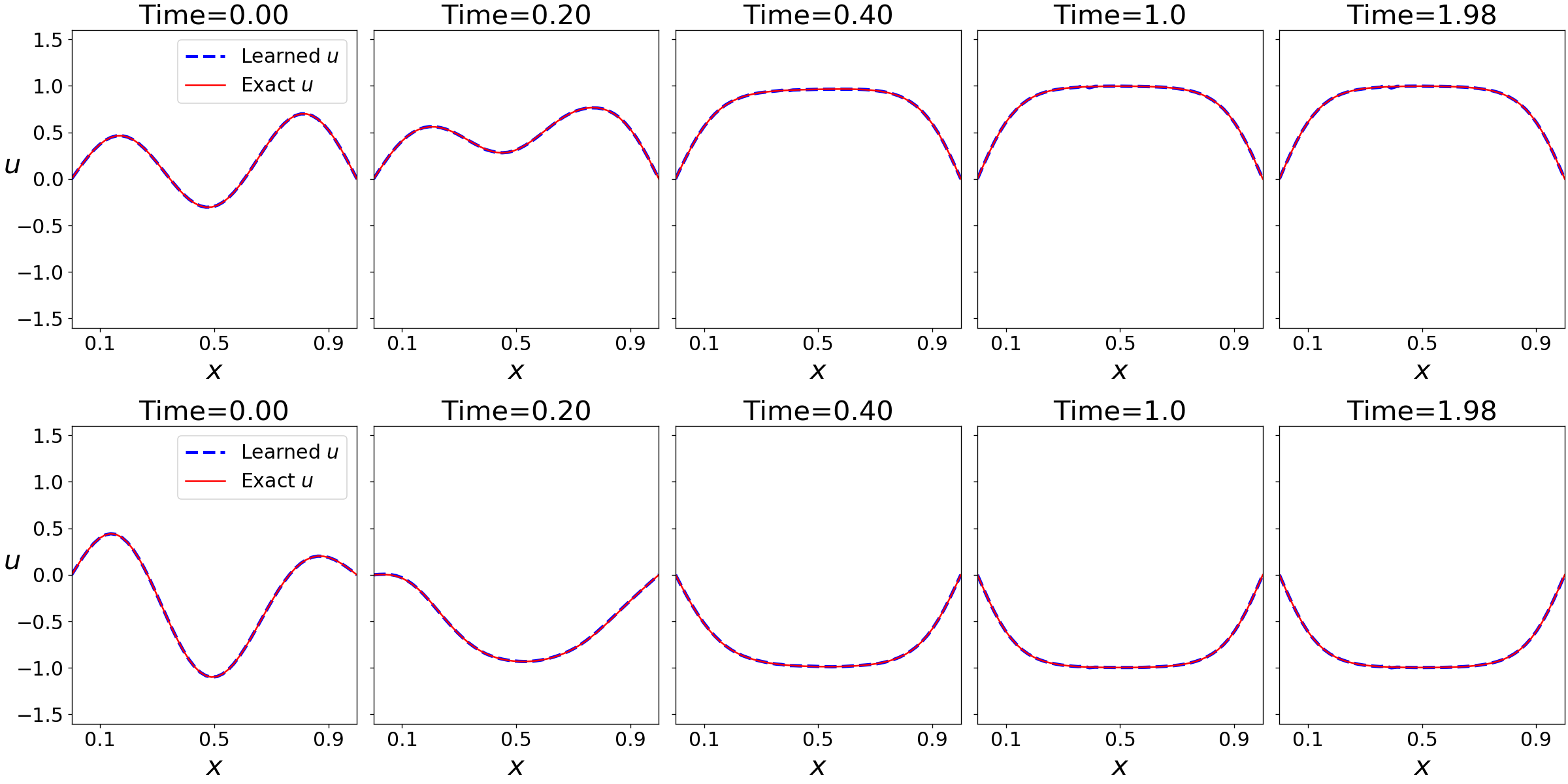}
	\caption{\ (Example 4): Comparison of trajectories of the learned dynamics and the original dynamics~\eqref{disct_GL} from different initial states.}
	\label{fig5}
\end{figure}

The number of discretization points is taken as $I=51$. We generate $10^4$ trajectories by solving the dynamics in~\eqref{disct_GL} starting from the initial states:
\begin{equation}
\begin{aligned}
u^0(x) = \frac{a\cdot\tilde{u}(x)}{\max_y \lvert \tilde{u}(y)\rvert},
\end{aligned}
\end{equation}
where $\tilde{u}(x)=\sum_{k=1}^{4}\hat{u}_k\sin(k\pi x)$ and $\{\hat{u}_k\}_{k=1}^4$, $a$ are drawn from the uniform distributions: $u_k\sim\mathcal{U}\left(-1,1\right)$, $a\sim\mathcal{U}\left(0,\frac{3}{2}\right)$.
Along each trajectory, we collect $200$ data points. In total, $X$ contains $2\times 10^6$ data points. Out of these data points, $76044$ representative data points are used to impose the orthogonality condition.

Fig.~\ref{fig5} shows a comparison of two trajectories of the learned dynamics and those of the original dynamics in the test dataset. The statistics (mean $\pm$ deviation) of the errors of the $1000$ trajectories is $1.220\times 10^{-2} \pm 7.734\times 10^{-2}$.
To assess the accuracy of the learned quasipotential, we compare $U_{\theta}$ and $U$ in~\eqref{QP_example4} along the minimum energy path (MEP) from $\vect{u}_{-}$ to $\vect{u}_{+}$. The MEP is computed using the string method~\cite{weinan2007simplified}. The comparison is shown in Fig.~\ref{fig6}, from which a good agreement can be observed. In particular, the learned quasipotential accurately captures the energy barrier between the two stable states.

\begin{figure}[t!]
	\centering
	\includegraphics[width=.5\linewidth]{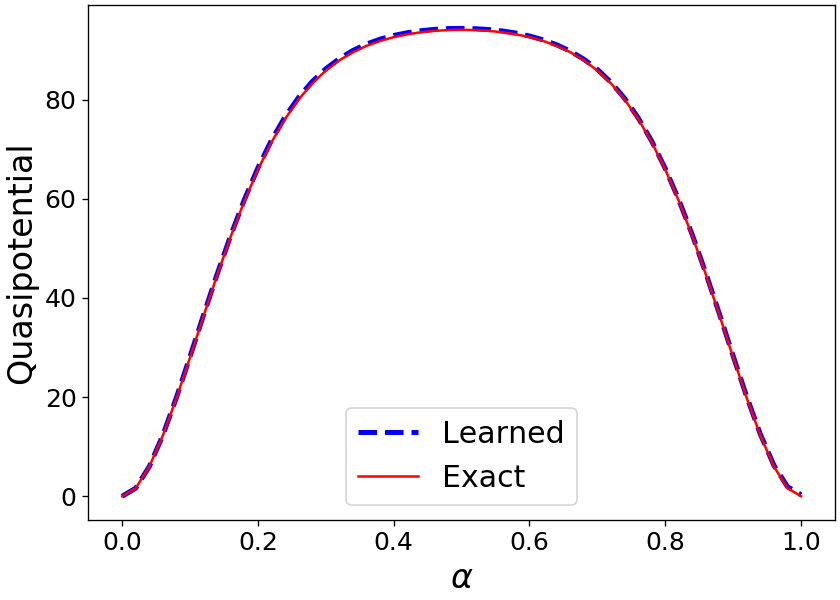}
	\caption{\ (Example 4): Comparison of the learned and exact quasipotentials for the discretized Ginzburg-Landau equation along the MEP, where $\alpha$ is the normalized arc-length parameter along the MEP.}
	\label{fig6}
\end{figure}
~\\
\noindent\textit{Example 5.} We consider the dynamics of the Brusselator on the spatial interval $[0,1]$,
\begin{equation}\label{Brus_PDE}
\begin{aligned}
u_t &= \frac{1}{\alpha}\left(u_{xx}+1+u^2v-(1+A)u\right),\\
v_t &= v_{xx}+Au-u^2v ,
\end{aligned}
\end{equation}
with the Neumann boundary conditions $u_x(0,t) = u_x(1,t) = 0$, $v_x(0,t) = v_x(1,t) = 0$,
and the initial condition $u(x,0)=u^0(x)$, $v(x,0)=v^0(x)$, 
where $\alpha$, $A$ are parameters. We discretize the interval $[0,1]$ with grid points $x_0,\dots,x_I$, where $x_i=ih$ and $h=1/I$. Then we approximate the spatial derivatives in Eq.~\eqref{Brus_PDE} using the central finite difference and obtain the following system of ODEs
\begin{equation}\label{disct_Bruss}
\begin{aligned}
\frac{d u_i}{dt} &= \frac{1}{\alpha}\left(\frac{u_{i-1}-2u_i+u_{i+1}}{h^2}+1+u_i^2v_i-(1+A)u_i\right),\\
\frac{d v_i}{dt} &= \frac{v_{i-1}-2v_i+v_{i+1}}{h^2}+Au_i-u_i^2v_i,\\
\end{aligned}
\end{equation}
for $0\leq i\leq I$, with the Neumann boundary conditions imposed by $u_{-1}=u_1$, $u_{I+1}=u_{I-1}$, $v_{-1}=v_1$, $v_{I+1}=v_{I-1}$,
and the initial condition $u_i(0)=u^0(x_i)$, $v_i(0)=v^0(x_i)$ for $0\leq i\leq I$,
where $(u_i,v_i)$ denotes the solution of Eq.~\eqref{Brus_PDE} at $x_i$. The state of the system is denoted by $\vect{x}=(u_0,...,u_I,v_0,...,v_I)$. The dynamics has a stable state: $u_i=1$, $v_i=A$ for $0\leq i\leq I$.

We take $\alpha=0.1$ and $A=0.5$. The number of discretization points is taken as $I=19$, so the discretized system is in the $40$-dimensional space. 
We generate $2\times 10^4$ trajectories by solving the dynamics in~\eqref{disct_Bruss} starting from the initial states:
\begin{equation}
\begin{aligned}
u^0(x) = \frac{a_1\cdot\tilde{u}(x)}{\max_y \lvert \tilde{u}(y)\rvert} +a_2,\quad
v^0(x) = \frac{a_3\cdot\tilde{v}(x)}{\max_y \lvert \tilde{v}(y)\rvert} +a_4,
\end{aligned}
\end{equation}
where $\tilde{u}(x)=\sum_{k=0}^4 \hat{u}_k\cos(k\pi x)$, $\tilde{v}(x)=\sum_{k=0}^4\hat{v}_k\cos(k\pi x)$ and $\{\hat{u}_k\}_{k=0}^4$, $\{\hat{v}_k\}_{k=0}^4$, $a_1$, $a_2$, $a_3$, $a_4$ are drawn from the uniform distributions
\begin{equation}
\begin{aligned}
    \hat{u}_k&\sim \mathcal{U}\left(-1,1\right),\ 
    \hat{v}_k\sim \mathcal{U}\left(-1,1\right),\ k=0,\dots,4,\\
    a_1&\sim \mathcal{U}\left(0,\frac{1}{2}\right),\ 
    a_2\sim \mathcal{U}\left(\frac{1}{2}+a_1,\frac{3}{2}-a_1\right),\ 
    a_3\sim \mathcal{U}\left(0,\frac{1}{2}\right),\ 
    a_4\sim \mathcal{U}\left(a_3,1-a_3\right).
\end{aligned}
\end{equation}
Along each trajectory, we collect $200$ data points. In total, $X$ contains $4\times 10^6$ data points. Out of these data points, $26218$ representative data points are used to impose the orthogonality condition.

Fig.~\ref{fig7} (upper panel) shows a comparison of one trajectory of the learned dynamics and that of the original dynamics in the test dataset. The statistics (mean $\pm$ deviation) of the errors of the $2000$ trajectories is $8.227\times 10^{-4} \pm 6.741\times 10^{-4}$.
The quasipotential is shown in Fig.~\ref{fig7} (lower panel) as a function of $(\hat{u}_0,\hat{v}_0)$, where $(\hat{u}_0,\hat{v}_0)$ corresponds to the state $u(x)\equiv\hat{u}_0$, $v(x)\equiv\hat{v}_0$ (left), and as a function of $(\hat{u}_1,\hat{v}_1)$, where $(\hat{u}_1,\hat{v}_1)$ corresponds to the state $u(x)=1+\hat{u}_1 \cos(\pi x)$, $v(x)=0.5+\hat{v}_1 \cos(\pi x)$ (right).
The numerical results agree well with those computed using the minimum action method~\cite{weinan2004minimum}.

\begin{figure}[t!]
	\centering
	\includegraphics[width=\linewidth]{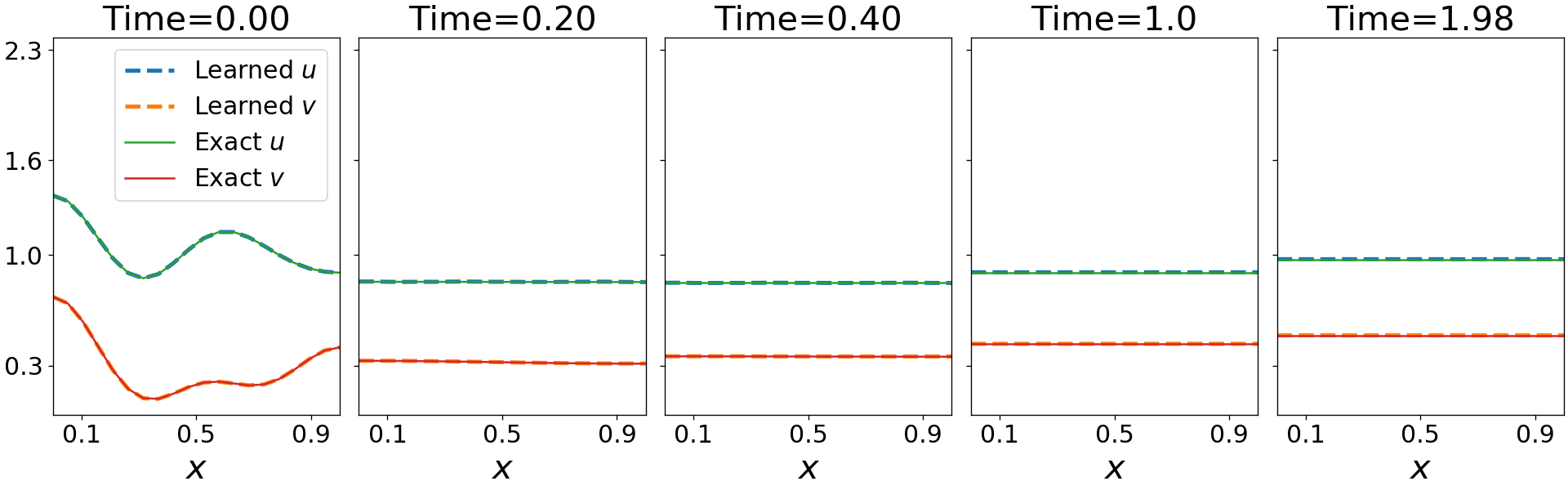}
	\includegraphics[width=0.8125\linewidth]{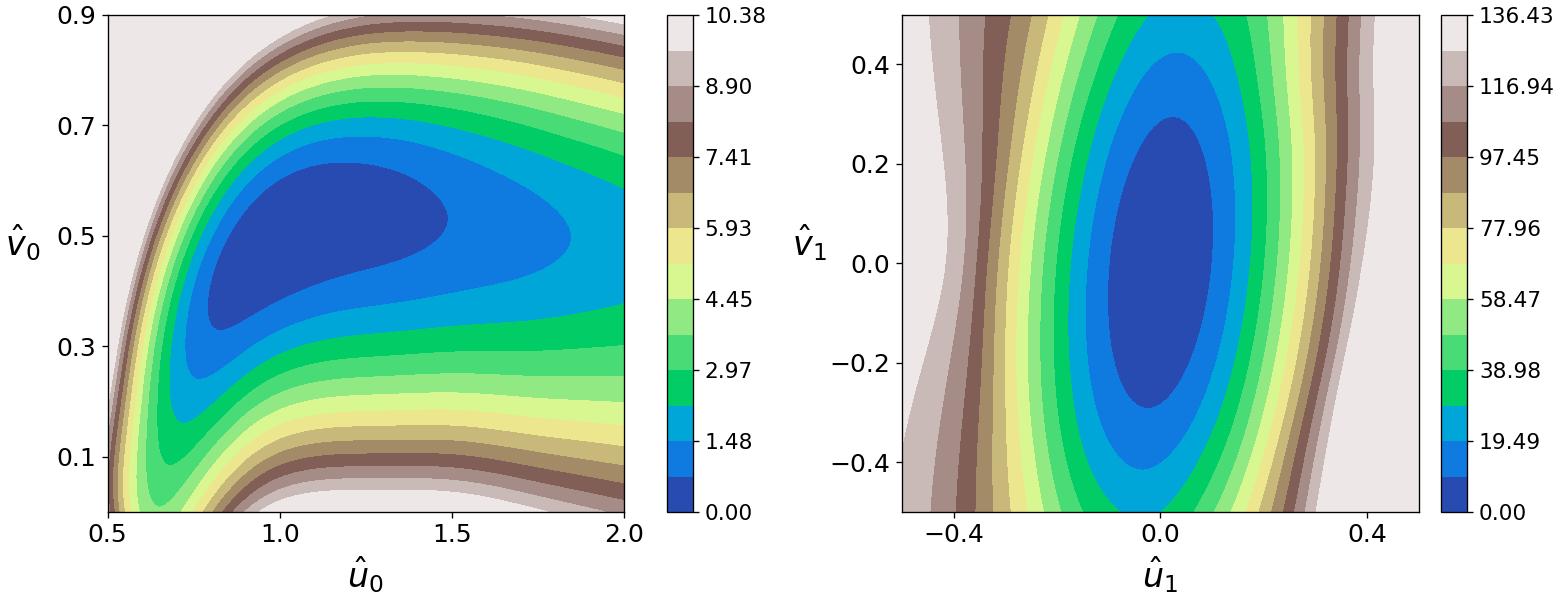}
	\caption{\ (Example 5): \textit{Upper Panel}:  Comparison of one trajectory of the learned dynamics and the original dynamics in~\eqref{disct_Bruss}. 
	\textit{Lower Panel}: Contour plots of the quasipotential as a function of $(\hat{u}_0,\hat{v}_0)$ for which the state is $u(x)\equiv\hat{u}_0$, $v(x)\equiv\hat{v}_0$ (left), and as a function of $(\hat{u}_1,\hat{v}_1)$ for which the state is $u(x)=1+\hat{u}_1 \cos(\pi x)$, $v(x)=0.5+\hat{v}_1 \cos(\pi x)$ (right).}
	\label{fig7}
\end{figure}

\section{Conclusion}\label{Conclusion}
In this paper, we proposed a method for computing the quasipotential for dynamical systems and at the same time learning the dynamics from the trajectory data.
This method is based on learning an orthogonal decomposition of the force field into potential and rotational components, each parameterized by a neural network. The neural networks are trained by minimizing a loss function composed of two parts: one is to reconstruct the dynamics and the other one is to impose the orthogonality condition between the potential and rotational components. The quasipotential associated with each attractor of the dynamical system can be obtained by confining the potential component to the corresponding basin of attraction.
We successfully applied the method to various examples including systems with stable equilibrium points, limit cycles and systems in high dimensions. 
The method is purely data driven in the sense that no explicit form of the dynamical system is required; in fact, an explicit model for the dynamics is learned from the observed trajectories in this method.
To the best of our knowledge, this is the first efficient and accurate method that can be used to map the landscape of the quasipotential in high dimensions.

After we obtain the quasipotential, we can compute other interesting objects associated with the dynamical system perturbed by small noise. For example, we can identify the minimum action path between the attractor $A$ and another state. Using the fact that the tangent of the path is parallel to $\vect{f}+\nabla U_A$ along the minimum action path, the path can be computed using the string method. The expected exit time from the basin of attraction can also be estimated using the minimum value of the quasipotential on the boundary of the basin of attraction.

In the current work, we demonstrated the effectiveness of the proposed method using examples with different features.
In the future, we plan to apply the method to problems of practical interest such as dynamical systems in fluid mechanics and biological systems.

\section*{Acknowledgements}
The work of Ren was supported in part by Singapore MOE AcRF grant R-146-
000-267-114, and the NSFC grant (No. 11871365). The work of QL was supported by the start-up grant at the National University of Singapore, under the PYP programme.

\end{document}